\documentclass[letterpaper,12pt]{amsart}
\title{Tverberg's theorem and graph coloring}

\author{Alexander Engstr\"om}
\address{Department of Mathematics and Systems Analysis \\
Aalto University, Helsinki, Finland}
\email{alexander.engstrom@aalto.fi}
\author{Patrik Nor\'en}
\address{Institute of Science and Technology Austria\\
Am Campus 1\\
3400 Klosterneuburg\\
Austria}
\email{patrik.noren@ist.ac.at}

\date{\today}

\usepackage{latexsym,array,delarray,amsthm,amssymb,epsfig}%eufrak
\usepackage{color}

\theoremstyle{plain}
\newtheorem{theorem}{Theorem}[section]
\newtheorem{lemma}[theorem]{Lemma}
\newtheorem{proposition}[theorem]{Proposition}
\newtheorem{corollary}[theorem]{Corollary}

\newtheorem*{theoremNoNumber}{Theorem}
\newtheorem*{conjectureNoNumber}{Conjecture}

\theoremstyle{definition}
\newtheorem{definition}[theorem]{Definition}

\theoremstyle{remark}
\newtheorem*{remark}{Remark}

\begin{document}

\begin{abstract}
The topological Tverberg theorem has been generalized in several directions by setting extra restrictions on the Tverberg partitions. 

Restricted Tverberg partitions, defined by the idea that certain points cannot be in the same part, are encoded with graphs. When two points are adjacent in the graph, they are not in the same part. If the restrictions are too harsh, then the topological Tverberg theorem fails. The colored Tverberg theorem corresponds to graphs constructed as disjoint unions of small complete graphs. Hell studied the case of paths and cycles.

In graph theory these partitions are usually viewed as graph colorings. As explored by Aharoni, Haxell, Meshulam and others there are fundamental connections between several notions of graph colorings and topological combinatorics. 

For ordinary graph colorings it is enough to require that the number of colors $q$ satisfy $q>\Delta,$ where $\Delta$ is the maximal degree of the graph. It was proven by the first author using equivariant topology that if $q>\Delta^2$ then the topological Tverberg theorem still works. It is conjectured that $q>K\Delta$ is also enough for some constant $K,$ and in this paper we prove a fixed-parameter version of that conjecture. 

The required topological connectivity results are proven with shellability, which also strengthens some previous partial results where the topological connectivity was proven with the nerve lemma.
\end{abstract}

\maketitle

\section{Introduction}

Tverberg's theorem \cite{tverberg1966} asserts that for any affine map $f$ from a simplex on $(d+1)(q-1)+1$ vertices to $\mathbb{R}^d$ there is a partition of the vertices into $q$ parts such that
\[ \bigcap_{i=1}^{q} f( \textrm{simplex spanned by part $i$}) \neq \emptyset. \]
It was generalized by B\'ar\'any, Schlosman and Sz\H{u}cs \cite{baranySchlosmanSzucs1981} to continuous $f$, but then the equivariant topology used in the proof requires $q$ to be a prime. Later this was extended to $q$ a prime power by \"Ozaydin \cite{ozaydin1987} (unpublished) and Volovikov \cite{volovikov1996}. 

Extra conditions on the Tverberg partitions can be encoded by graphs, indicating that there are many Tverberg partitions, as done by Hell \cite{hell2007,H2}. Part of his work was extended by  Engstr\"om \cite{engstrom2011} who proved the following theorem.

\begin{theoremNoNumber}
Let $G$ be a graph with $(d+1)(q-1)+1$ vertices and $q$ a prime power satisfying
\[ q> \max_{v\in V(G)} ( |N^2(v)| + 2|N(v)| ) \]
where $N^2(v)$ is the set of vertices on distance two from $v$ and $N(v)$ is the set of vertices adjacent to $v$.
Then for any continuous map $f$ from a simplex with the same vertex set as $G$ to $\mathbb{R}^d$
there is a $q$-coloring of $G$ such that
\[ \bigcap_{i=1}^{q} f( \textrm{simplex spanned by color $i$}) \neq \emptyset. \]
\end{theoremNoNumber}

The equivariant topology used to prove that theorem builds on that certain spaces are enough topologically connected. That was proven by topological methods, as the nerve lemma. But the question was raised, if, as was done for the chessboard complexes by Zieger \cite{ziegler1994}, this could be proven by vertex decomposability and shellability. We prove that this is possible in Corollary~\ref{cor:DF1}.

With the previous known versions of Tverberg's theorem the following natural conjecture was made in \cite{engstrom2011}.

\begin{conjectureNoNumber}\label{conj:main}
There is a constant $K$ such that the following holds:
Let $G$ be a graph on $(d+1)(q-1)+1$ vertices and maximal degree $\Delta$, and let $f$ be a continuous map from a simplex $\Sigma$ with the same vertex set as $G$ to $\mathbb{R}^d$. If
\[  q> K \Delta \] 
then there is a $q$-coloring of $G$ satisfying
\[ \bigcap_{i=1}^{q} f( \textrm{simplex spanned by color $i$}) \neq \emptyset. \]
\end{conjectureNoNumber}
The emeritus of the field, Helge Tverberg, believes in the conjecture \cite{tverberg2011}.
In Corollary~\ref{corollary:mainCor} we prove the following fixed-parameter version of it.
\begin{theoremNoNumber}
For every $\varepsilon >0$ there exists a constant $K_\varepsilon$ such that the following holds:
Let $G$ be a graph on $((d+1)(q-1)+1)(1+\varepsilon)$ vertices and maximal degree $\Delta$ (with $d$ and $\Delta$  large enough depending on $\varepsilon$), and let $f$ be a continuous map from a simplex $\Sigma$ with the same vertex set as $G$ to $\mathbb{R}^d$. If
\[  q> K_\varepsilon \Delta \] 
then there is a $q$-coloring of $G$ satisfying
\[ \bigcap_{i=1}^{q} f( \textrm{simplex spanned by color $i$}) \neq \emptyset. \]
\end{theoremNoNumber}

The crucial statements in equivariant topology of Section~\ref{sec:equi} builds on graphs being vertex decomposable. In Section~\ref{sec:vdg} we introduce this concept and prove some fairly technical statements about it. We have made an effort to make Section~\ref{sec:vdg} completely independent and only about graph theory, allowing experts in this field to improve on our results without a deep understanding of the equivariant topology used in Section~\ref{sec:equi}.

\subsection{Some notation}

The neighborhood $N^\circ_G(v)$ in a graph $G$ of a vertex $v$ is the set of vertices of $G$ adjacent to $v$; and $N^\bullet_G(v)=N^\circ_G(v) \cup \{v\}.$ The vertices on distance two from $v$ in $G$, $N_G^2(v)$, are all vertices $u$ with a path on two edges to $v$. Usually we drop the $G$ subscript if the graph containment is clear. The complete graph on $q$ vertices is $K_q.$

\section{Decomposing graphs}\label{sec:vdg}

\subsection{Vertex decomposability of simplicial complexes}

In topological combinatorics a central notion is \emph{shellability}. A simplicial complex is shellable if its facets can be pealed off in a controlled manner, providing a certificate that the space topologically is a collection of equidimensional spheres wedged together at a point. One method to prove a complex shellable is by the stronger notion of \emph{vertex decomposable}. 
It is a powerful method, employed for example by Provan and Billera for independence complexes of matroids \cite{new01}; and by Lee for the associahedron \cite{new02} as explained by Jonsson in \cite{new03}. Most simplicial complexes studied in topological combinatorics are not wedges of spheres and a good bound on their topological connectivity is the best attainable description of their homology. One way to achieve that is to prove that a pure skeleton is vertex decomposable, as done for example by Ziegler \cite{ziegler1994} for chessboard complexes.
For independence complexes determined by graphs, we introduce a filtrated version of vertex decomposable right off on the level of graphs, and then return to its topological interpretation and consequences in Section~\ref{sec:equi}. Note that we do not discuss the elementary question regarding when one-dimensional simplicial complexes viewed as graphs are vertex decomposable.

\subsection{Vertex decomposability of graphs}

\begin{definition}\label{def:vertexDecomposableGraphs}
For every non-negative integer $k$ we define the graph property $\mathtt{VD}_k$.
Any graph $G$ is $\mathtt{VD}_0$, and a graph $G$ on $k$ vertices and no edges is $\mathtt{VD}_k$.
If $G$ is a graph with a vertex $v$ such that $G \setminus v$ is $\mathtt{VD}_k$ and $G\setminus N^\bullet(v)$ is $\mathtt{VD}_{k-1},$ then
$G$ is $\mathtt{VD}_k$.
\end{definition}

\begin{remark}
The empty graph is $\mathtt{VD}_0$.
\end{remark}

This proposition is included to give some elementary examples.

\begin{proposition}
If $G$ is the disjoint union of $k$ edges and $l$ vertices, then $G$ is $\mathtt{VD}_{k+l}.$
\end{proposition}
\begin{proof}
For $k=0$ this is true by definition. For $k>0$ pick a vertex $v$ that is not isolated. Then
 $G \setminus v$ is $\mathtt{VD}_{k+l}$ and $G\setminus N^\bullet(v)$ is $\mathtt{VD}_{k+l-1}$ by induction, and $G$ is $\mathtt{VD}_{k+l}$ by the definition.
\end{proof}

\begin{proposition}
If $G$ is $\mathtt{VD}_k$ and $k\geq l \geq 0$ then $G$ is also $\mathtt{VD}_l$.
\end{proposition}

\begin{proof}
Assume that $l>0$ since any graph is $\mathtt{VD}_0.$ In the case of only isolated vertices, apply the recursive definition several times instead of the first part of the definition right off.
For the remaining cases it follows by induction on the number of vertices. 
\end{proof}

\begin{remark}
In Proposition~\ref{prop:isVD} in Section~\ref{sec:equi} it will be proven that the $(k-1)$-skeleton of the independence complex of $G$ is pure $(k-1)$-dimensional and vertex decomposable if $G$ is $\mathtt{VD}_k$.
\end{remark}

\begin{proposition}\label{prop:disjointunion}
If $G$ is $\mathtt{VD}_k$ and $H$ is $\mathtt{VD}_l$ then the disjoint union of $G$ and $H$ is $\mathtt{VD}_{k+l}$.
\end{proposition}
\begin{proof}
This is proved by induction on $|V(G\sqcup H)|$ and $|E(G\sqcup H)|$. For the case $|E(G\cup H)|=0$ this is true by definition and when $|V(G\sqcup H)|=0$ then $|E(G\sqcup H)|=0$.

Without loss of generality assume that $G$ has an edge.

There is a vertex $v\in V(G)$ so that $G\setminus v$ is $\mathtt{VD}_k$ and $G\setminus N^\bullet(v)$ is $\mathtt{VD}_{k-1}$.

Now $G\setminus v \sqcup H=(G\sqcup H)\setminus v$ and $G\setminus N^\bullet(v)\sqcup H=(G\sqcup H)\setminus N^\bullet(v)$. By induction it follows that $(G\sqcup H)\setminus v$ is $\mathtt{VD}_{k+l}$ and that $(G\sqcup H)\setminus N^\bullet(v)$ is $\mathtt{VD}_{k+l-1}$. This proves that the disjoint union of $G$ and $H$ is $\mathtt{VD}_{k+l}$.
\end{proof}

Our goal in preparation of Section~\ref{sec:equi} and the equivariant topology, is to prove that  graphs are $\mathtt{VD}_k$ for as high $k$ as possible. There is a procedure that is not strong enough, but since our approach builds on it, we explain it. First we need a lemma that in the simplicial complex setting is due to Ziegler \cite{ziegler1994}. The lemma needed is a special case of Proposition~\ref{prop:disjointunion}.
\begin{lemma}\label{lemma:isolatedVD}
If $G$ has an isolated vertex $v$ and $G\setminus v$ is $\mathtt{VD}_{k-1}$, then $G$ is $\mathtt{VD}_k.$
\end{lemma}

\begin{proof}
This is the special case of Proposition~\ref{prop:disjointunion} when one of graphs is $G\setminus v$ and the other graph is $v$.
\end{proof}

Lemma~\ref{lemma:isolatedVD} indicates that one way to recursively prove that a graph is $\mathtt{VD}_k$ for a non-trivial $k$, is to turn vertices isolated by removing their adjacent vertices, and then increase $k$ by applying Lemma~\ref{lemma:isolatedVD}. Here is one way to formalize that.

\begin{lemma}\label{lemma:massiveVD}
Let $G$ be a graph with a vertex $v$ whose neighborhood is $N(v)=\{ u_1, u_2, \ldots, u_n \}$. If $G \setminus N^\bullet (v)$ and
\[ G \setminus ( N^\bullet(u_i) \cup  \{u_1,u_2, \ldots, u_{i-1}\} ) \textrm{ for }1\leq i \leq n   \]
are  $\mathtt{VD}_{k-1}$, then $G$ is $\mathtt{VD}_k$.
\end{lemma}
\begin{proof} 
From Lemma~\ref{lemma:isolatedVD} and that $G \setminus N^\bullet (v)$ is $\mathtt{VD}_{k-1}$ we get $G \setminus N^\circ(v) = G\setminus \{u_1,u_2, \ldots, u_n\}$ is $\mathtt{VD}_{k}$. Now the idea is to add the vertices $u_n, u_{n-1}, \ldots, u_1$ one by one to get $G$ and control the invariant $\mathtt{VD}_k$ during the process.

For $i=n,n-1,\ldots, 2,1,$ use Definition~\ref{def:vertexDecomposableGraphs}  on $G\setminus \{u_1,u_2, \ldots u_{i-1} \} $ with the vertex $u_i$. 
It follows that $G\setminus \{u_1,u_2, \ldots u_{i-1} \} $ is $\mathtt{VD}_k$  from that $(G\setminus \{u_1,u_2, \ldots u_{i-1} \}) \setminus u_i =  G\setminus \{u_1,u_2, \ldots u_{i} \}$ is $\mathtt{VD}_k$ and $(G\setminus \{u_1,u_2, \ldots u_{i-1} \}) \setminus N^\bullet (u_i)  = G\setminus ( N^\bullet (u_i) \cup \{u_1,u_2, \ldots u_{i-1} \}) $ is $\mathtt{VD}_{k-1}.$

With the last step of $i=1$, we add the vertex $u_1$ and get $G\setminus \{u_1,u_2, \ldots u_{i-1} \}=G$ which is $\mathtt{VD}_k$.
\end{proof}

For generic graphs, avoiding global structures as in cartesian products, the following proposition is efficient.

\begin{proposition}[Dochtermann \& Engstr\"om~\cite{dochtermannEngstrom2009}, Theorem 5.9]
Let $G$ be a graph on $n$ vertices and maximal degree $\Delta>0$. Then $G$ is $\mathtt{VD}_{\lfloor n/2\Delta \rfloor}.$
\end{proposition}

\begin{proof}
We do induction on the number of vertices. The basis case of the induction is when the number of vertices are $0 \leq n < 2\Delta.$ In that range the proposition states that $G$ should be 
$\mathtt{VD}_0,$ and all graphs satisfy that.

If $n\geq 2\Delta$ then fix some vertex $v$ of $G$ with neighborhood $N^\circ(v)=\{ u_1, u_2, \ldots, u_m \}$. Now consider the following subgraphs: $G \setminus N^\bullet (v)$ and  $G \setminus ( N^\bullet (u_i) \cup  \{u_1,u_2, \ldots, u_{i-1}\} )$ for $1\leq i \leq m.$ All of these graphs includes a subgraph of $G$ gotten by deleting an edge and all neighbors of the vertices of that edge. So, all of them have less vertices than $G$, but the difference is at most $2\Delta$ vertices. Thus by induction, and by the fact that the maximal degree never increases by taking subgraphs, all of them are $\mathtt{VD}_{\lfloor n/2\Delta \rfloor -1}.$ By Lemma~\ref{lemma:massiveVD}, the graph $G$ is $\mathtt{VD}_{\lfloor n/2\Delta \rfloor}.$
\end{proof}

\subsection{A few algorithms}
\begin{figure}
 \begin{center}
  \includegraphics[width=160mm]{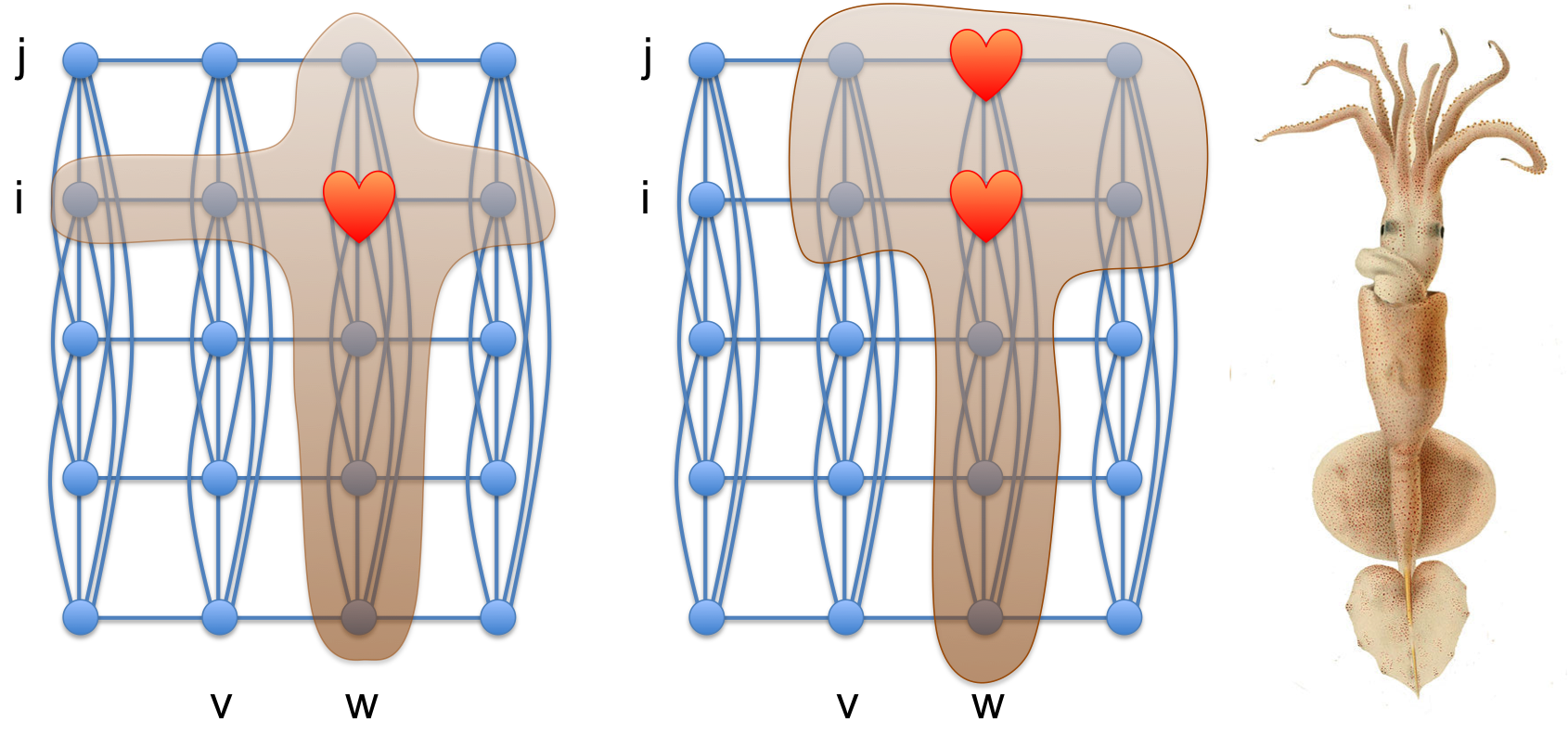}
 \caption{For $G$ a path on four vertices and $q=5$, any subset of the marked vertices is a squid in $G\square K_q$ with body $w$. On the right is a grimalditeuthis bonplandi squid without tentacles \cite{albert1900}.}\label{fig:Squids}
 \end{center} 
 \end{figure}
In Engstr\"om \cite{engstrom2011} a much weaker version of our main theorems was proved by removing squids. Our approach follows this idea, but is much more technically involved. To begin with we define a class of algorithms to remove squids, called \emph{DF-algorithms}. Then we prove that any DF-algorithm provides certificates that graphs are of the right 
$\mathtt{VD}_k$ class. 

But first we define the cartesian product and squids. The \emph{cartesian product} of two graphs $G$ and $H$, denoted $G\square H$, is the graph with vertex set $V(G)\times V(H)$ and edge set
\[ \{ (u,v)(u',v) \mid uu' \in E(G), v\in V(H) \} \cup \{ (u,v)(u,v') \mid u \in V(G), vv'\in E(H) \}. \] 
As an example, the cartesian product of the graph consisting of $k$ isolated vertices and the edge $K_2$ is $k$ isolated edges, a graph that is $\mathtt{VD}_k$. When passing to independence complexes, an important class of graphs are cartesian products of complete graphs, because they become chessboard complexes.

\begin{definition}
A \emph{squid} with \emph{body} $w$ in $G\square K_q$ is a subset of $V(G\square K_q)$ that  is either
\begin{itemize}
\item[(i)] a \emph{subset} of
     \[ (N^\circ_G(v)\cup N^\circ_G(w)) \times\{i\} \cup \{w\}\times\{1,2,\ldots, q\} \]
for two adjacent vertices $v$ and $w$, and $1\leq i\leq q$, or
\item[(ii)] a \emph{subset} of
     \[ N^\circ_G(w)\times\{i,j\} \cup \{w\}\times\{1,2,\ldots, q\} \]
where $1\leq i<j\leq q$.
\end{itemize}
The vertices not of the form $(w,k)$ are \emph{arms}. The \emph{heart} of a squid of type (i) is $(w,i)$ and the hearts of a squid of type (ii) are $(w,i)$ and $(w,j)$. The hearts and body is part of the squid data, and two squids could be on the same subset of $V(G\square K_q)$ but differ in that regard.
\end{definition}

 If $S$ is a squid, then we also use the symbol $S$ for the subset of $V(G\square K_q)$ in set theoretic statements if no confusion occurs. In Figure~\ref{fig:Squids} are examples of squids.
 An instance of squids removed from a cartesian product $G\square K_q$ is modeled as a \emph{DF-tuple}.

\begin{definition}
A \emph{DF-tuple} is a five tuple $(G,q,j,\{S_1,S_2,\ldots,S_j\},m)$ consisting of
\begin{itemize}
\item[(1)] a finite graph $G$ with vertices in $\mathbb{N};$
\item[(2)] integers $|G|\geq m\geq j\geq 0,$ and $q>0$; and
\item[(3)] squids $S_1, S_2, \ldots,S_j$ in $G\square K_q$.
\end{itemize}
\end{definition}

In the definition of DF-tuples nothing is assumed regarding if squids intersect each other or are empty. An adversary would try to achieve the opposite, to cover the cartesian product with as few squids as possible. A particularly bad situation would be if a whole copy of $K_q$ would be covered by arms of squids without anyone having its heart there. To avoid this we construct DF-algorithms. A DF-algorithm is a collection of DF-tuples with an instruction for how to remove one more squid if $j<m$. The squid to be removed is defined by a map from the collection of tuples into itself. Alternatively, we could have stated this as a decision-tree where the player trying to maximize $k$ in $\mathtt{VD}_{k}$ decides where  one heart of the squid should be, and the adversary decides on what type of squid with that heart that should be removed.

\begin{definition}\label{def:DF-algorithm}
A \emph{DF-algorithm} $(\mathbf{A},\mathcal{G})$ is a set $\mathcal{G}$ of DF-tuples, and a map
\[\mathbf{A}:\{(G,q,j,\{S_1,S_2\ldots,S_j\},m)\in\mathcal{G}|j<m \}\rightarrow \mathbb{N}\times\mathbb{N}\]
that for any $T=(G,q,j,\{S_1,S_2,\ldots,S_j\},m)\in\mathcal{G}$ with $j<m$,
satisfies 
\begin{itemize}
\item[(1)] $(v,i):=\mathbf{A}(T) \in V(H)$ where  $H=G\square K_q\setminus \cup^j_{i=1}{S_j}$, and
\item[(2)] if 
\begin{itemize}
\item[(a)] $S\subseteq N^\circ_H(v,i)\cup N^\circ_H(v,j')$ for some $(v,j')\in H,$ or 
\item[(b)] $S\subseteq(N^\circ_H(v,i)\cap G\square\{i\})\cup N^\circ_H(u,i)$ for some $u\in N^\circ_G(v)$ with $(u,i)\in H,$
\end{itemize}
then $(G,q,j+1,\{S_1,S_2,\ldots,S_j,S\},m)\in \mathcal{G}.$
\end{itemize}
\end{definition}

After setting up the definitions and notations for removing squids with DF-algorithms, we now prove that they certify that the relevant cartesian products are $\mathtt{VD}_k$.

\begin{theorem}\label{thm:VD}
Let $(\mathbf{A},\mathcal{G})$ be a DF-algorithm. If $(G,q,j, \{S_1,S_2,\ldots,S_j\},m)$ is in $\mathcal{G}$ then $G \square K_q \setminus \cup_{i=1}^j S_i$  is $\mathtt{VD}_{(m-j)}$.
\end{theorem}

\begin{proof}
Set $H= G \square K_q \setminus \cup_{i=1}^j S_i $. The proof is by induction on $m-j$. The base case $m=j$, that $H$ is $\mathtt{VD}_{0}$, follows from Definition~\ref{def:vertexDecomposableGraphs}.

Now assume that $m>j$ and set $(v,i') = \mathbf{A}((G,q,j,\{S_1,S_2,\ldots,S_j\},m))$. The neighbors of $(v,i')$ in $H$ are either in $G \times \{i'\}$ or in 
$ \{v \} \times K_q $. Chose a linear order of the neighbors
\[ N_H^\circ(v,i') = \{ (u_1,j_1), (u_2, j_2), \ldots, (u_n, j_n) \} \]
such that $u_1 = u_2 = \cdots u_k = v$ and $j_{k+1}=j_{k+2}= \cdots j_n = i'$ for some $k$.

For $l=1,2, \ldots, n$ define squids
\[ S'_l = N^\circ_H(u_l,j_l) \cup \{ (u_1,j_1), (u_2, j_2), \ldots, (u_l, j_l) \}. \]
By just parsing the definition of a DF-algorithm letter by letter in this situation, we see that $(G,q,j+1,\{S_1,S_2,\ldots,S_j,S'_l\},m)$ is in $\mathcal{G}$
\begin{itemize}
\item[] by (2.b) in Definition~\ref{def:DF-algorithm} for $1\leq l \leq k,$ and
\item[] by (2.a) in Definition~\ref{def:DF-algorithm} for $k< l \leq n.$
\end{itemize}
Define one more squid $S'' = N^\circ_H(v,i') \cap G\times \{i'\}$ and once again by just parsing (2.a) of Definition~\ref{def:DF-algorithm} letter by letter in this situation, we find that
$(G,q,j+1,\{S_1,S_2,\ldots,S_j,S''\},m)$ is in $\mathcal{G}.$ This is not an unexpected consequence of Definition~\ref{def:DF-algorithm}, rather the other way around. That definition was constructed to be able to prove this theorem with exactly this proof. The interested reader might simply reverse engineer Definition~\ref{def:DF-algorithm} from this proof.
There is nothing deep going on here, just formal verifications.

By induction, $H\setminus S''$, and $H \setminus S'_l$ for $1\leq l \leq n$, are $\mathtt{VD}_{(m-j-1)}$. We can now conclude by 
Lemma~\ref{lemma:massiveVD} that $H=G \square K_q \setminus \cup_{i=1}^j S_i$ is $\mathtt{VD}_{(m-j)}$.
\end{proof}

\begin{corollary}\label{cor:VD}
Let $(\mathbf{A},\mathcal{G})$ be a DF-algorithm. If
$(G,q,0, \emptyset ,m) \in \mathcal{G}$ then 
$G \square K_q $ is $\mathtt{VD}_{m}$.
\end{corollary}
\begin{proof}
This is a special case of Theorem~\ref{thm:VD}.
\end{proof}

We now introduce two DF-algorithms. Using the first one, we later show the same Tverberg type results as in Engstr\"om \cite{engstrom2011}, but employ only the combinatorial topology of shellability instead of stronger abstract tools from algebraic topology. This proves Conjecture 3.10 of \cite{engstrom2011}, and gives a result in the same spirit as Ziegler's paper \cite{ziegler1994}, where he proved that the optimal connectivity bounds of chessboard complexes can be proved by shelling skeletons of chessboard complexes.

\begin{theorem}\label{thm:DF1}
Fix a graph $G$ and a positive integer $m$ with $m\le |G|$. Let $q$ be an integer with $q>|N^2(v)| + 2|N^\circ(v)|$ for all vertices $v$ of $G.$

Let $\mathcal{G}$ be the set of DF-tuples $(G,q,j,\{S_1,S_2,\ldots,S_j\},m)$.

Then $G \square K_q \setminus \cup_{i=1}^j S_i$ is non-empty if $j<m$ and any map
\[\mathbf{A}:\{(G,q,j,\{S_1,S_2\ldots,S_j\},m)\in\mathcal{G}|j<m \}\rightarrow \mathbb{N}\times\mathbb{N}\]
sending $(G,q,j,\{S_1,S_2,\ldots,S_j\},m)$ to any vertex of $G \square K_q \setminus \cup_{i=1}^j S_i$ defines a DF-algorithm $(\mathbf{A},\mathcal{G}).$
\end{theorem}
\begin{proof}
The first step is to prove that $H=G \square K_q \setminus \cup_{i=1}^j S_i$ is non-empty. By assumption $j<m$ and there is a vertex in $v$ in $G$ that is not a body of a squid $S_i$. We claim that $v\square K_q\cap H$ is non-empty. If it was empty, it was deleted by arms of squids.

The worst case is if all vertices in $N(v)$ are bodies of type (ii) squids in $\{S_1,S_2,\ldots,S_j\}$ and all vertices in  $N^2(v)$ are bodies of type (i) squids in $\{S_1,S_2,\ldots,S_j\}$. In this case the maximal number of vertices removed from $v\square K_q$ is $|N^2(v)|+2|N(v)|$, but $q>|N^2(v)|+2|N(v)|$ and then $v\square K_q\cap H$ is non-empty.

Now a map
\[\mathbf{A}:\{(G,q,j,\{S_1,S_2\ldots,S_j\},m)\in\mathcal{G}|j<m \}\rightarrow \mathbb{N}\times\mathbb{N}\]
sending $(G,q,j,\{S_1,S_2,\ldots,S_j\},m)$ to any vertex of $G \square K_q \setminus \cup_{i=1}^j S_i$ defines a DF-algorithm $(\mathbf{A},\mathcal{G}),$
as there is no restrictions on the squids.
\end{proof}

\begin{corollary}\label{cor:DF1}
Let $q$ be an integer and $G$ a graph on $m$ vertices with $q>|N^2(v)| + 2|N(v)|$ for all vertices $v.$ Then $G \square K_q $ is $\mathtt{VD}_{m}$.
\end{corollary}
\begin{proof} There are no restrictions on the collections of squids in $\mathcal{G}$ from Theorem~\ref{thm:DF1} and then $(G,q,0, \emptyset ,m) \in \mathcal{G}$. Corollary~\ref{cor:VD} now proves the statement.
\end{proof}

To prove the second main theorem of this paper, we need a more dynamic way to remove squids. We will use the following strategy to remove squids from $G \square K_q$:
We first remove $n_1$ squids with hearts on the top row $G \times {r_1}$ where $r_1=1$. The removal of these squids will have different effect on the rows  $G \times {j}$ with $j>1$. If a large number of squids have arms also on row $G \times {j}$, then this row is a bad choice for continuing the removal of squids from. So the next step is to let $r_2$ be the top-most row with the most number of preserved vertices. We remove $n_2$ squids with hearts on the row $G \times {r_2}$ and proceed in the same manner, until $n_1+n_2+ \cdots + n_k$ is large enough. To ensure that we simply don't run out of vertices, the sizes $n_i$ are specified with a dynamic DF-size scheme.

\begin{definition}\label{def:dDFscheme}
Let $n,q,\Delta$ be positive integers and let $(n_1,n_2,\ldots, n_k)$ be a sequence of positive integers.

A tuple $(n,q,\Delta,(n_1,\ldots,n_k))$is a \emph{dynamic DF-size scheme} if
\begin{itemize}
\item[(1)] $q \geq k>0$ and all $n_i>0$; and
\item[(2)] for each $1 \leq j \leq k$
\[ 
   \left( \frac{ \Delta }{q-j+1} + 1 \right) \left(\sum_{i=1}^{j-1} n_i\right) + 2\Delta n_j  \leq n.
\]
\end{itemize}
\end{definition}

\begin{theorem}\label{thm:goodEpsilon}
For every $\varepsilon >0$ there exists a constant $K_\varepsilon$ such that for every graph $G$ with $N(1+\varepsilon)$ vertices  (with $N$�and $\Delta$ large enough depending on $\varepsilon$)  and \[  q> K_\varepsilon \Delta, \]
there is a dynamic DF-size scheme $(n_1,n_2,\ldots, n_k)$ with $n=N(1+\varepsilon)$ and $N\leq \sum_{i=1}^kn_i$.
\end{theorem}
\begin{proof}
We only need asymptotic estimates and disregard that several of the variables should be integers. To satisfy (2) of Definition~\ref{def:dDFscheme} we prove that 
 \[  a \left(\sum_{i=1}^{j-1} n_i\right) + 2\Delta n_j  \leq n = N(1+\varepsilon)  \]
 for some $a$ when the $n_j$ are defined properly. To satisfy this inequality, with equality for all $j$, we set
 \[ n_j = \frac{N(1+\varepsilon)}{2\Delta}\left( \frac{2\Delta - a}{2\Delta} \right)^{j-1}. \]
 Now set $k=2\Delta \gamma$ and $a=\sqrt{1 + \varepsilon}$ in
 \[
 \begin{array}{rcl}
\displaystyle \sum_{j=1}^ks_j 
	&= & \displaystyle \frac{N(1+\varepsilon)}{2\Delta}  \frac{1-\left( \frac{2\Delta - a}{2\Delta} \right)^{k}}{1-\left( \frac{2\Delta - a}{2\Delta} \right)}  \\
	&= & \displaystyle N(1+\varepsilon) \frac{1-\left( \frac{2\Delta - a}{2\Delta} \right)^{k}}{a} \\
	&= & \displaystyle N\sqrt{1+\varepsilon} \left( 1-\left( 1- \frac{\gamma \sqrt{1+\varepsilon}}{2\Delta\gamma}  \right)^{2\Delta \gamma} \right) \\
	&\geq & \displaystyle N\sqrt{1+\varepsilon} \left( 1- e^{\gamma \sqrt{1+\varepsilon} } \right) \\
	& = & N \\
\end{array}
 \]
with $\gamma = - \frac{1}{\sqrt{1+\varepsilon}} \ln \left(1 -  \frac{1}{\sqrt{1+\varepsilon}} \right).$ Finally, the variable $a$ should satisfy
 \[ \sqrt{1 + \varepsilon} = a = 1 + \frac{\Delta}{q-k} = 1 + \frac{1}{K_\varepsilon - 2\gamma},  \]
 and we set
 \[ K_\varepsilon = \sqrt{1 + \varepsilon} -1 + 2\gamma =   \sqrt{1 + \varepsilon} -1 - \frac{2}{\sqrt{1+\varepsilon}} \ln \left(1 -  \frac{1}{\sqrt{1+\varepsilon}} \right).  \]
 \end{proof}

Now we describe how to get a DF-algorithm from a dynamic DF-size scheme.

\begin{definition}\label{def:dDFs}
Given a graph $G$ with vertices in $\mathbb{N}$ of maximal degree $\Delta$, and a dynamic DF-size scheme $(n_1,n_2,\ldots, n_k)$ with $n,q$;
the \emph{dynamic DF-scheme} is the set $\mathcal{G}$ of DF-tuples $(G,q,j,\{S_1,S_2,\ldots,S_j\},n)$ such that:
\begin{itemize}
\item for each $l$ with $s_1+s_2+\ldots+s_{l-1} \leq j$ all the squids
\[ S_{n_1+n_2+\ldots+n_{l-1}+1}, \ldots, S_{ \max\{n_1+n_2+\cdots +n_l,j\}} \]
have hearts on the same row $G \times r_l$,
\item all the $r_l$ are different,
\item when the squids with hearts on rows $G\times r_1, G\times r_2, \ldots, G \times r_{l-1}$ are deleted, then $G \times r_l$ is the top-most row with maximal number of preserved vertices.
\end{itemize}
together with a map
\[\mathbf{A}:\{(G,q,j,\{S_1,S_2\ldots,S_j,\},n)\in\mathcal{G}|j<n \}\rightarrow \mathbb{N}\times\mathbb{N}\]
defined as the vertex $(v,r_i)\in G\square K_q \setminus (S_1 \cup S_2 \cup \cdots \cup S_j)$ for which $n_1+n_2+\cdots+n_{i-1}< j \leq  n_1+n_2+\cdots+n_{i}$ and
\[ v = \min ( u \mid (u,r_i) \in G\square K_q \setminus (S_1 \cup S_2 \cup \cdots \cup S_j) .\]
\end{definition}

A dynamic DF-scheme is a DF-algorithm, since the dynamic DF-scheme guarantees 
\[  \{ u \mid (u,r_i) \in G\square K_q \setminus (S_1 \cup S_2 \cup \cdots \cup S_j \} \]
to be non-empty.

\begin{corollary}\label{cor:DF2}
For every $\varepsilon >0$ there exists a constant $K_\varepsilon$ such that for every graph $G$ with $N(1+\varepsilon)$ vertices (with $N$�and $\Delta$ large enough depending on $\varepsilon$) $G\square K_q$ 
is $\mathtt{VD}_N$ if $q> K_\varepsilon \Delta$
\end{corollary}
\begin{proof}
This follows directly from Corollary~\ref{cor:VD}, Theorem~\ref{thm:goodEpsilon} and Definition~\ref{def:dDFs}.
\end{proof}

\section{Equivariant Topology}\label{sec:equi}

In this section we will use the facts about vertex decomposable graphs derived in Section~\ref{sec:vdg} to derive new theorems of Tverberg type.
Recall that a set of vertices of a graph $G$ is independent if none of them are adjacent. The \emph{independence complex} of a graph $G$, denoted $\mathtt{Ind}(G)$, is the simplicial complex on the same vertex set as $G$ whose faces are the independent sets of $G$. For basic combinatorial topology we refer to Bj\"orner's excellent survey~\cite{bjorner1995}, but we collect a few useful facts. The link of a vertex $v$ of $\Sigma$ is $\mathrm{lk}_\Sigma(v)=\{\sigma \in \Sigma\ \mid v\not\in \sigma,\,\, \sigma \cup \{v\} \in \Sigma \}$, and the deletion of $v$ is $\mathrm{dl}_\Sigma(v)=\Sigma\setminus v = \{ \sigma \in \Sigma \mid v \not\in \sigma \}.$ For independence complexes $\mathrm{lk}_{\mathtt{Ind}(G)}(v)=\mathtt{Ind}(G\setminus N^\bullet(v))$ and $\mathrm{dl}_{\mathtt{Ind}(G)}(v)=\mathtt{Ind}(G\setminus v).$ A more comprehensive introduction to basic operations on independence complexes is given in~\cite{engstrom2008}. The $k$-skeleton of $\Sigma$ is $\Sigma^{\leq k} = \{ \sigma \in \Sigma \mid \mathrm{dim}\, \sigma \leq k \},$ and an easy exercise is
$\mathrm{lk}_{\Sigma^{\leq k}}(v) = \mathrm{lk}_{\Sigma}(v)^{\leq k-1}$ and  $\mathrm{dl}_{\Sigma^{\leq k}}(v) = \mathrm{dl}_{\Sigma}(v)^{\leq k}.$

\begin{definition}\label{def:VD}
A simplicial complex $\Sigma$ is \emph{vertex decomposable} if it is pure, and either $\Sigma = \{\emptyset\}$ or it has a vertex $v$ with $\mathrm{lk}_\Sigma(v)$ and $\mathrm{dl}_\Sigma(v)$ vertex decomposable.
\end{definition}

The most important consequences of a pure $d$-dimensional complex being vertex decomposable, is that it is shellable, homotopically a wedge of $d$-dimensional spheres, and, in particular, $(d-1)$-connected.

\begin{proposition}\label{prop:isVD}
If $G$ is a $\mathtt{VD}_k$ graph then $\mathtt{Ind}(G)^{\leq k-1}$ is pure $(k-1)$-dimensional and 
vertex decomposable.
\end{proposition}
\begin{proof}
We first prove that if $G$ is $\mathtt{VD}_{k}$ then $\mathtt{Ind}(G)^{\leq k-1}$ is pure $(k-1)$-dimensional.

The first case is that $\mathtt{Ind}(G)^{\leq -1} = \{ \emptyset \}$ is pure $(-1)$-dimensional for all $G$.

The second case is when $G$ is a $k$-vertex graph without edges. Then $\mathtt{Ind}(G)^{\leq k-1}$ is a $(k-1)$-simplex and pure $(k-1)$-dimensional.

The third case is when $G$ is $\mathtt{VD}_{k}$ since $G\setminus v$ is $\mathtt{VD}_{k}$ and $G\setminus N^\bullet(v)$ is $\mathtt{VD}_{k-1}.$ Say that $\sigma \in \mathtt{Ind}(G)^{\leq k-1}$ would be a facet of dimension less than $k-1$ to reach a contradiction. If $v\not\in \sigma$ then we get a contradiction right off since $\sigma$ is in the pure $(k-1)$-dimensional complex $\mathtt{Ind}(G \setminus v)^{\leq k-1}.$ If $v\in \sigma$, then $\sigma \setminus v$ is not a facet of $\mathtt{Ind}(G \setminus N^\bullet(v))^{\leq k-2}$ since it is pure and $(k-2)$-dimensional. If we extend $\sigma \setminus v$ to a facet $\tau$ in $\mathtt{Ind}(G \setminus N^\bullet(v))^{\leq k-2}$, then $\sigma$ is strictly included in the facet $\tau \cup \{v\}$ of $\mathtt{Ind}(G)^{\leq k-1}$ and we have a contradiction.

Now we prove that $\mathtt{Ind}(G)^{\leq k-1}$ is vertex decomposable if $G$ is $\mathtt{VD}_{k}$.

The complex $\{ \emptyset \}$ is vertex decomposable by definition, and simplices are by an easy argument left to the reader.

Now to the case that $G$ is $\mathtt{VD}_{k}$ since $G\setminus v$ is $\mathtt{VD}_{k}$ and $G\setminus N^\bullet(v)$ is $\mathtt{VD}_{k-1}.$
The deletion $\mathrm{dl}_{\mathtt{Ind}(G)^{\leq k-1}}(v) = \mathrm{dl}_{\mathtt{Ind}(G)}(v)^{\leq k-1}= \mathtt{Ind}(G\setminus v)^{\leq k-1}$ is vertex decomposable since 
 $G\setminus v$ is $\mathtt{VD}_{k}$. The link $\mathrm{lk}_{\mathtt{Ind}(G)^{\leq k-1}}(v) = \mathrm{lk}_{\mathtt{Ind}(G)}(v)^{\leq k-2}= \mathtt{Ind}(G\setminus N^\bullet(v))^{\leq k-2}$
 is vertex decomposable since $G\setminus N^\bullet(v)$ is $\mathtt{VD}_{k-1}.$ We conclude that $\mathtt{Ind}(G)^{\leq k-1}$ is vertex decomposable.
\end{proof}

\begin{theorem}\label{theorem:mainEqui}
Let $q\geq 2$ be a prime power, $d\geq 1,$ and set $N=(d+1)(q-1)+1$. Let $\Sigma$ be a simplex on the same vertex set as $G$ and $f$ a continuous function from $\Sigma$ to $\mathbb{R}^d$.
If $G\square K_q$ is $\mathtt{VD}_{N}$, then there is a $q$-coloring of $G$
\[ C_1 \cup C_2 \cup \cdots \cup C_q = V(G) \]
such that 
\[ \bigcap_{i=1}^q f(\textrm{\emph{simplex spanned by} $C_i$}) \]
 is non-empty.
\end{theorem}
\begin{proof}
The complex $\mathtt{Ind}(G \square K_q)^{\leq N-1}$ is vertex decomposable by Proposition~\ref{prop:isVD} since $G\square K_q$ is $\mathtt{VD}_{N}$.
The complex $\mathtt{Ind}(G \square K_q)$ is $(N-2)$-connected since $\mathtt{Ind}(G \square K_q)^{\leq N-1}$ is that.

Now the remaining part of the proof is standard equivariant topology, a minor modification of Theorem 2.2 in \cite{engstrom2011}, and we only sketch the proof. 

The map $f$ from $\Sigma$ to $\mathbb{R}^d$ induces a map $f^{\ast q}$ from the $q$-fold join 
$\Sigma^{\ast q}$ to the $q$-fold join $(\mathbb{R}^d)^{\ast q}$. If we restrict $\Sigma^{\ast q}$ to the
$\sigma_1 \ast \sigma_2 \ast  \cdots \ast \sigma_q$ where all pairs $\sigma_i,\sigma_j$ are disjoint, then we get the 2-wise $q$-fold deleted join
\[ \mathtt{Ind}(G' \square K_q) \]
where $G'$ is the graph on the same vertex set as $G$ but with no edges. If we further restrict the deleted join to require that all $\sigma_i$ are independent sets, then we get
\[ \mathtt{Ind}(G \square K_q). \]

To prove the theorem by contradiction, suppose that there is no $q$-coloring whose images of the faces given by the colors intersect in a non-empty set. Then the image of the map can be restricted, and we have a map
\[ f^{\ast q} : \mathtt{Ind}(G \square K_q) \rightarrow (\mathbb{R}^d)^{\ast q} \setminus \{ \gamma_1\mathbf{x} + \cdots + \gamma_q\mathbf{x} \mid \mathbf{x} \in \mathbb{R}^d \}. \]
By assumption $q$ is a prime power $p^k$, and there is a free $\mathbb{Z}_p^k$ action on $\mathtt{Ind}(G \square K_q)$ and $(\mathbb{R}^d)^{\ast q}$ by permuting the $q$ coordinates. This action extends to the map $f^{\ast q}$.
By a Borsuk-Ulam type argument of Volovikov~\cite{volovikov1996}, such an equivariant map into $ (\mathbb{R}^d)^{\ast q} \setminus \{ \gamma_1\mathbf{x} + \cdots + \gamma_q\mathbf{x} \mid \mathbf{x} \in \mathbb{R}^d \}$ forces the connectivity of $\mathtt{Ind}(G \square K_q)$ to be at most $N-3=(d+1)(q-1)-2$.  But since it is $(N-2)$-connected we have a contradiction.
\end{proof}

\begin{corollary}\label{corollary:mainCor}
For every $\varepsilon >0$ there exists a constant $K_\varepsilon$ such that the following holds:
Let $G$ be a graph on $((d+1)(q-1)+1)(1+\varepsilon)$ vertices and maximal degree $\Delta$ (with $d$ and $\Delta$ are large enough depending on $\varepsilon$), and let $f$ be a continuous map from a simplex $\Sigma$ with the same vertex set as $G$ to $\mathbb{R}^d$. If
\[  q> K_\varepsilon \Delta \] 
then there is a $q$-coloring of $G$ satisfying
\[ \bigcap_{i=1}^{q} f( \textrm{simplex spanned by color $i$}) \neq \emptyset. \]
\end{corollary}
Bertrand's postulate states that there is a prime between $q$ and $2q.$
According to classical analytic number theory, there is a prime between $q$ and $q+q^\alpha$ for some $\alpha<1$ when $q$ large enough. A contemporary result is $\alpha=0.525$ \cite{baker2001}. In the proof it is only needed that for every $\delta>0,$ if $q$ is sufficiently large, there is a prime between $q$ and $q+\delta q$. 

\begin{proof}
The proof is in two steps. The first step is to prove it for $q$ a prime power, the second step is to prove it for general $q$ using the prime power case and estimates for the density of primes.

According to Corollary~\ref{cor:DF2} there is a constant $K^p_\varepsilon$ such that $G\square K_{q}$ is $\mathtt{VD}_{(d+1)(q-1)+1}$
if $q> K^p_\varepsilon \Delta.$ By Theorem 3.3 we see that there is a $q$-coloring with the intersection of the images of monochromatic simplices non-empty if $q$ is a prime power.

An important easy property of the numbers $K^p_\varepsilon$ needed in the following argument, which has not been spelled out explicitly before, is that if $\varepsilon_1<\varepsilon_2$ then $K^p_{\varepsilon_1}\ge K^p_{\varepsilon_2}$. To ensure that all involved numbers are integers there is a lower bound $\Delta_\varepsilon$ for $\Delta$.

The next step is to construct $K_\varepsilon$ that works for arbitrary $q$. It will be proved that $K_\varepsilon=\max(K^p_{\varepsilon/16},B_{\varepsilon/4}/\Delta_\varepsilon)$ works.

For every $\delta>0$ there is an integer $B_\delta$ so that if $q\ge B_\delta$ then there is a prime $q_p$ so that $q\le q_p<q(1+\delta)$.

Assume that $q\ge \max(K^p_{\varepsilon/16},B_{\varepsilon/4}/\Delta_\varepsilon)\Delta$. Let $q_p\ge q$ be the prime power closest to $q$. Now $q_p$ is bounded above by $q(1+\varepsilon/4)$. 

It is possible to find $\varepsilon_p$ so that
\[
((d+1)(q-1)+1)(1+\varepsilon)=((d+1)(q_p-1)+1)(1+\varepsilon_p)
\]
and as $q_p$ is bounded above by $q(1+\varepsilon/4)$ a straightforward calculation shows that $\varepsilon_p$ is bounded below by $\varepsilon/16$.

Now there is a $q_p$ coloring of $G$ that gives a non-empty intersection by the prime power case. Only intersecting $q$ of the color classes also give a non-empty intersection. One can extend the partial coloring obtained by the $q$ picked classes into a complete coloring only using $q$ colors as the maximum degree of a vertex is less than $q$, this is true as $K_\varepsilon$ can be assumed to be greater than $1$. This new coloring also give a non-empty intersection as the intersection only grows by adding vertices to the color classes.
\end{proof}

\section*{Acknowledgement}\label{sec:acknowledgement}

Alexander Engstr\"om was a Miller Research Fellow at UC Berkeley, and gratefully acknowledges support from the Miller Institute for Basic Research in Science. Patrik Nor\'en gratefully acknowledges support from the Wallenberg foundation.

\end{document}